\documentclass[a4, 12pt, 
]{amsart}
\usepackage{mathtools}
\usepackage{enumerate}
\usepackage{amssymb}
\usepackage{amsmath}
\usepackage{amsthm}
\usepackage{enumitem}
\setlist[enumerate,2]{
  ref=\alph*,
}
\usepackage{hyperref}
\usepackage{braket}
\hypersetup{
           breaklinks=true,   % splits links across lines
           colorlinks=true,   % displays links as colored text instead of blocks
        }
\numberwithin{equation}{section}
\theoremstyle{plain}
\newtheorem{thm}[equation]{Theorem}

\newtheorem{prop}[equation]{Proposition}
\newtheorem{lem}[equation]{Lemma}
\newtheorem{conj}[equation]{Conjecture}
\newtheorem*{claim*}{Claim}
\theoremstyle{definition}
\newtheorem{defn}[equation]{Definition}
\newtheorem{setup}[equation]{Setup}
\theoremstyle{remark}
\newtheorem{rem}[equation]{Remark}

\newcommand{\delb}{\sqrt{-1}\partial\bar{\partial}}

\usepackage{xparse}%interval
\usepackage{bbm}
\author[R.\ Murakami]{Rei Murakami}
\address{Mathematical Institute, Tohoku University, 6-3, Aramaki Aza-Aoba, Aoba-ku, Sendai 980-8578, Japan}
\email{rei.murakami.p3@dc.tohoku.ac.jp, reimurakami66@gmail.com}

\begin{document}

\title[Numerical criteria on the complex Hessian quotient equations]{Numerical criteria on the complex Hessian quotient equations with the Calabi symmetry}

\begin{abstract}
    Assuming Calabi symmetry, we prove that a numerical condition ensures the solvability of the complex Hessian quotient equation, as conjectured by Sz{\'e}kelyhidi. We also propose a conjecture on the existence of a $k$-subharmonic representative in a given cohomology class and confirm it under the assumption of Calabi symmetry or when the class is semiample.
\end{abstract}

\maketitle

\section{Introduction}
This paper studies a variant of partial positivity of cohomology classes, originating from the theory of complex Hessian quotient equations. 
Let $(X,\omega)$ be an $n$-dimensional compact K{\"a}hler manifold and $\alpha$ a closed real $(1,1)$-form on $X$.
For two integers $k$ and $\ell$ satisfying $n\ge k > \ell\ge0$, we consider the complex Hessian quotient equation which is defined by the following PDE for an unknown function $\varphi$:
\begin{equation}\label{eq: cHe}
    \alpha_\varphi^k\wedge\omega^{n-k}=e^H \, \alpha_\varphi^\ell\wedge\omega^{n-\ell}.
\end{equation}
Here $H$ is a smooth function satisfying the condition $\int_X\alpha^k\wedge\omega^{n-k}=\int_X e^H \alpha^\ell\wedge\omega^{n-\ell}$ and $\alpha_\varphi$ denotes the form $\alpha+\delb\varphi$. The study of complex Hessian equations originated in the paper \cite{CNS}, which addressed Dirichlet problems on real domains. In complex geometry, several examples of \eqref{eq: cHe} have appeared, playing important roles in various studies, such as the complex Monge-Amp{\`e}re equation ($k=n$ and $\ell=0$) and the $J$-equation ($k=n$ and $\ell=n-1$).
 
In \cite{Sze}, Sz{\'e}kelyhidi estimated the $C^2$-norm of solutions of \eqref{eq: cHe} by introducing the notion of $\mathcal{C}$-subsolutions, and then applied the continuity method to prove the following result:
\begin{thm}[{\cite[Propositions 21 and 22]{Sze}}]\label{thm: sze}
    Let $(X,\omega)$ be an $n$-dimensional compact K{\"a}hler manifold and $\alpha$ a closed real $(1,1)$-form on $X$. Assume that there exists a strictly $\alpha$-$(\omega,k)$-subharmonic function \textup{(}see the definition below\textup{)}. 
    \begin{enumerate}[font=\normalfont]
        \item\label{item: cHe} Suppose that $\ell=0$. Then there exists a strictly $\alpha$-$(\omega,k)$-subharmonic solution $\varphi$ of \eqref{eq: cHe}.
        \item\label{item: sze} Suppose that $\ell\ge1$ and that $e^H$ is a constant $c$. Then the following conditions are equivalent:
        \begin{enumerate}[font=\normalfont]
        \item\label{item: csub} There exists a strictly $\alpha$-$(\omega,k)$-subharmonic function $\varphi$ such that 
        $$k\alpha_\varphi^{k-1}\wedge\omega^{n-k}-c\ell\alpha_\varphi^{\ell-1}\wedge\omega^{n-\ell}>0$$ 
        as an $(n-1,n-1)$-form. 
        \item There exists a strictly $\alpha$-$(\omega,k)$-subharmonic solution $\varphi$ of \eqref{eq: cHe}.
        \end{enumerate}
        \end{enumerate}
\end{thm}

\begin{defn}[$\alpha$-$(\omega,k)$-subharmonicity]
    Let $(X,\omega)$ be an $n$-dimensional compact K{\"a}hler manifold and $\alpha$ a closed real $(1,1)$-form on $X$. A smooth function $\varphi$ is said to be \textit{strictly $\alpha$-$(\omega,k)$-subharmonic} if $\alpha_\varphi^i\wedge\omega^{n-i}>0$ for any $0\le i \le k$.
\end{defn}

When $k=n$ and $\ell=0$, Theorem \ref{thm: sze} corresponds to Yau's theorem \cite{Yau}.
The case $\ell=0$ (known as the complex $k$-Hessian equation) and the case $k=n$ were proved by \cite{DK} and \cite{SW, FLM}, respectively. 

It is natural to ask when there exists a strictly $\alpha$-$(\omega,k)$-subharmonic function and when condition \eqref{item: csub} in Theorem \ref{thm: sze} holds. 
Motivated by a conjecture on the $J$-equation proposed in \cite{LS} (later proved in \cite{GChen, Datar-Pingali, Song}), Sz{\'e}kelyhidi formulated a conjecture \cite[Conjecture 23]{Sze} (Conjecture \ref{conj} \eqref{item: conj2} below) to characterize condition \eqref{item: csub} in Theorem \ref{thm: sze}.
To state a conjecture regarding the existence of a strictly $\alpha$-$(\omega,k)$-subharmonic function, we introduce the following notations:
\begin{defn}
    Let $(X,\omega)$ be an $n$-dimensional compact K{\"a}hler manifold and $k$ an integer with $0< k\le n$. The sets $\mathcal{K}_{k,\omega}$ and $\mathcal{P}_{k,\omega}$ are defined by
    \begin{align*} 
    \mathcal{K}_{k,\omega}&=\left\{[\alpha]\in {H}^{1,1}(X,\mathbb{R})\ \middle\vert \text{ there exists a strictly $\alpha$-$(\omega,k)$-subharmonic function}
    \right\},\\
    \mathcal{P}_{k,\omega}&=\left\{[\beta]\in H^{1,1}(X,\mathbb{R})\ \middle\vert \begin{array}{l}
              \int_V \beta^{k-n+p}\wedge\omega^{n-k}>0 
              \text{ for any $p$-dimensional} 
              \\
              \text{subvariety $V$ of $X$ with $p\ge n-k+1$}
        \end{array}\right\}.
    \end{align*}
\end{defn}

\begin{conj}\label{conj}
    Let $(X,\omega)$ be an $n$-dimensional compact K{\"a}hler manifold and $\alpha$ a closed real $(1,1)$-form.
    \begin{enumerate}[font=\normalfont]
        \item\label{item: conj} $[\alpha]$ is in $\mathcal{K}_{k,\omega}$ if and only if there exists a path in $\mathcal{P}_{k,\omega}$ passing through $[\alpha]$ and $\mathcal{K}_{k,\omega}$.
        \item\label{item: conj2} Suppose that $[\alpha]\in \mathcal{K}_{k,\omega}$. Then, there exists a strictly $\alpha$-$(\omega,k)$-subharmonic function that satisfies $k \alpha_\varphi^{k-1}\wedge\omega^{n-k}-c\ell\alpha_\varphi^{\ell-1}\wedge\omega^{n-\ell}>0$ if and only if for any $p$-dimensional subvariety $V$ with $p\ge n-\ell$, we have
        \begin{align*}
            \int_V\frac{k!}{(k-n+p)!}\alpha^{k-n+p}\wedge\omega^{n-k}-c\, \frac{\ell!}{(\ell-n+p)!}\alpha^{\ell-n+p}\wedge\omega^{n-\ell}>0.
        \end{align*}
    \end{enumerate}
\end{conj}

\begin{rem}\mbox{}
    \begin{enumerate}[label=(\alph*)]
        \item Conjecture \eqref{item: conj} is a generalization of the case $k=n$ studied in \cite{DP} (the Nakai-Moishezon criterion when $X$ is projective).
        \item  A natural path to consider in Conjecture \eqref{item: conj} is $[\alpha]+t[\omega]$. If the conjecture is true, testing on this path and finding some path as stated in the conjecture should be equivalent.
        \item  When $X$ is projective, Conjecture \eqref{item: conj} suggests that $\mathcal{K}_{k,\omega}=\mathcal{P}_{k,\omega}$ by testing a path of the form $[\alpha]+t c_1(L)$, where $L$ is an ample line bundle.
        \item Suppose that $X$ is projective, $k=2$, and $l=1$. Since the equation \eqref{eq: cHe} can be rewritten as
        \begin{equation*}
            \left(\alpha_\varphi-\frac{c}{2}\omega\right)^2\wedge\omega^{n-2}=\frac{c^2}{4}\omega^n,
        \end{equation*}
        Conjecture \eqref{item: conj2} follows from Conjecture \eqref{item: conj}. 
    \end{enumerate}
\end{rem}

In this paper, we confirm the conjecture in some cases using the Calabi ansatz, widely known in complex differential geometry since the work \cite{Cal}.
We freely use the following notations:
\begin{setup}\label{setup}
    Let $X:=\mathbb{P}(\mathcal{O}_M\oplus L^{\oplus (m+1)})\rightarrow M$ be the projective bundle over an $n$-dimensional compact K{\"a}hler manifold $M$, where $L$ is a negative line bundle over $M$. Define the subvariety $P_0$ and the divisor $D_\infty$ by $$P_0:=\mathbb{P}(\mathcal{O}_M)\subset X \text{ and } D_\infty:=\mathbb{P}(L^{\oplus(m+1)})\subset X$$ and define the K{\"a}hler form $\omega_M$ on $M$ by $$\omega_M:=-F_h=\sqrt{-1}\partial\bar{\partial}\log h,$$ where $F_h$ is the curvature of a Hermitian metric $h$ on $L$. Let $\omega\in[\pi^*\omega_M]+b\eta$ be a K{\"a}her form on $X$ which can be written as $\pi^*\omega_M+\sqrt{-1}\partial\bar{\partial}u(\rho)$ on $X\setminus (P_0\cup D_\infty)$, , where $\eta$ is the Poincaré dual of the subvariety $D_\infty$, a function $\rho$ is defined by $\rho(z,w)=\log|(z,w)|^2_h$ for $(z,w)\in X$, and a function $u$ is a smooth function on $\mathbb{R}$.
\end{setup}

\begin{thm}\label{thm: main}
    Suppose Setup \ref{setup}. Let $p$ and $q$ be real numbers. Then, Conjecture \ref{conj} is true for the cohomology class $[\alpha]=p[\pi^*\omega_M]+q\eta$.
\end{thm}

When $k=\dim X$, Theorem \ref{thm: main} was proved by \cite{FL} (see also \cite{DMS}). However, it is not straightforward to generalize their result to the cases $k<\dim X$, since their proof uses the property that $[\alpha]$ is K{\"a}hler in the case $k=\dim X$. 
In addition, \cite{Jacob} proved the solvability of harmonic polynomials, generalizing the case of the deformed Hermitian-Yang-Mills equation \cite{JS}, under certain numerical conditions. 
Regarding Conjecture \ref{conj} \eqref{item: conj2}, the proof of Theorem \ref{thm: main} shows that the subvarieties which possibly violate the numerical condition are only $P_0$ and $D_\infty$, while \cite{FL, DMS} show that only $P_0$ can violate a numerical condition. We do not know whether the possibility of $D_\infty$ can also be eliminated in Theorem \ref{thm: main}.  
Moreover, \cite{FL, DMS} analyzed the long-time behavior of the corresponding flow even when the numerical condition does not hold. We have not found analogous results in our cases $k<\dim X$.

We prove Conjecture \ref{conj} \eqref{item: conj} in semiample cases as well.

\begin{thm}\label{thm: semiample}
    Conjecture \ref{conj} \eqref{item: conj} is true when the cohomology class $[\alpha]$ is the first Chern class of a semiample line bundle $L$.
\end{thm}

The paper is organized as follows. In Section \ref{sec: pre}, we review some well-known properties of elementary symmetric functions for the readers' convenience. In Section \ref{sec: calabi}, we prove Theorem \ref{thm: main}.
In Section \ref{sec: semi}, we prove Theorem \ref{thm: semiample} based on the arguments in \cite[Section 3]{Mat} with small modifications. In Appendix A, we make some remarks on the relation between $(\omega,k)$-subharmonicity and $(\dim X-k)$-positivity in the sense of Andreotti-Grauert.

\begin{subsection}*{Acknowledgements}
    The author wishes to express his gratitude to his advisor, Shin-ichi Matsumura, for his invaluable assistance in writing this paper, for suggesting the study of semiample cases, and for his constant support. He also thanks Ramesh Mete for helpful comments and discussions, and the anonymous referee for valuable comments.
    This work was supported by JSPS KAKENHI Grant Number JP24KJ0346.
\end{subsection}

\section{Preliminaries}\label{sec: pre}
In this section, for the reader's convenience, following \cite[Lecture 2]{Spr}, we briefly summarize the properties of the $k$-th elementary symmetric function which we will use later.

\begin{defn}Let $k$ be a non-negative integer.
    \begin{enumerate}
        \item The $k$-th elementary symmetric function $\sigma_k:\mathbb{R}^n\to\mathbb{R}$ is defined by
        \begin{equation*}
        \sigma_k(\lambda)=\sum_{i_1<i_2<\dots<i_k}{\lambda_{i_1}}{\lambda_{i_2}}\dots{\lambda_{i_k}}.
        \end{equation*}
        We define $\sigma_0(\lambda)=1$ for any $\lambda\in\mathbb{R}^n$.
        \item The admissible cone $\Gamma_k$ of the $k$-th elementary symmetric function $\sigma_k$ is defined by
        \begin{equation*}
            \Gamma_k:=\left\{\lambda\in\mathbb{R}^n\mid\sigma_i(\lambda)>0 \text{ for any $i$ such that $0\le i \le k$}\right\}.
        \end{equation*}
    \end{enumerate}
\end{defn}
Note that when $k=n$, the admissible cone $\Gamma_n$ is nothing but the positive cone 
\begin{equation*}
    \{\lambda\in\mathbb{R}^n\mid\lambda_i>0\text{ for any $i$}\}.
\end{equation*}
The closure of the positive cone is 
$
    \bar{\Gamma}_n:=\{\lambda\in\mathbb{R}^n\mid\lambda_i\ge0 \text{ for any $i$}\}.
$
A basic property of $\sigma_k$ is the monotonicity in $\Gamma_k$, which follows from the fact that $\sigma_k$ is a hyperbolic polynomial (see \cite[p9--10]{Spr}).

\begin{prop}[{\cite[Corollary 2.4 and Theorem 2.16]{Spr}}]\label{prop: mono}\mbox{}
    \begin{enumerate}[font=\normalfont]
        \item The $k$-th elementary symmetric function is strictly increasing in $\Gamma_k$, namely, for any $\lambda\in\Gamma_k$ and $\lambda'\in\lambda+\bar{\Gamma}_n\setminus\{0\}$, we have $\sigma_k(\lambda')>\sigma_k(\lambda)$.
        \item Let $0\le \ell<k \le n$. The function $$\left(\frac{\sigma_k\ \Big/{n\choose k}}{\sigma_\ell\ \Big/{n\choose \ell}}\right)^{\frac{1}{k-\ell}}$$ is strictly increasing in $\Gamma_k\subset\mathbb{R}^n$.
    \end{enumerate}
\end{prop}
\begin{prop}[The Newton inequality, {\cite[Proposition 2.7]{Spr}}]\label{prop: Newton}
    Let $\lambda\in\mathbb{R}^n$ and $r$ an integer with $1\le r \le n$. Then, we have
    \begin{equation*}
            \frac{\sigma_{r-1}(\lambda)}{{n\choose {r-1}}}\cdot \frac{\sigma_{r+1}(\lambda)}{{n\choose{r+1}}}\le\left(\frac{\sigma_r(\lambda)}{{n\choose r}}\right)^2.
    \end{equation*}
    Furthermore, the equality holds if and only if all the components of $\lambda$ are equal. In particular, if $\sigma_r(\lambda)\neq0$, then we have
    $
            \sigma_{r-1}(\lambda)\sigma_{r+1}(\lambda)<\sigma^2_r(\lambda).
    $
    
\end{prop}

\section{Proof of Theorem \ref{thm: main}}\label{sec: calabi}
\subsection{Set-up}
In this section, as calculated by \cite{FL}, we express complex Hessian equations
\begin{equation}\label{eq: cHe'}
    \alpha^k\wedge\omega^{\dim X-k}=c \alpha^l\wedge\omega^{\dim X-\ell},
\end{equation}
as ordinary differential equations via the Calabi ansatz technique. We use the same notations as in Setup \ref{setup}. Let $v$ be a smooth function on $\mathbb{R}$ and consider a closed real form $\alpha:=p\pi^*\omega_M+\delb v(\rho)$ on $X\setminus(P_0\cup D_\infty)$, where $p\in\mathbb{R}$.
By \cite{Cal} (see also \cite{DMS}), the form $\alpha$ can be extended to a closed real form on $X$ if 
\begin{itemize}
    \item $V_0(r):=v(\log r)$ can be extended smoothly across $r=0$,
    \item $V_\infty(r):=v(-\log r)+q\log r$ can be extended smoothly across $r=0$, where $q\in\mathbb{R}$.
\end{itemize}
Moreover, if these conditions are met, the cohomology class of the resulting form, denoted also by $\alpha$, is $p[\pi^*\omega_M]+q\eta$. Also, for the function $u$ in Setup \ref{setup}, the following additional conditions are met since the resulting form $\omega$ is K{\"a}hler:
\begin{itemize}
    \item $u'>0$ and $u''>0$,
    \item $U_0(r):=u(\log r)$ satisfies $U_0'(0)>0$,
    \item $b>0$.
\end{itemize}
In particular, we have 
\begin{align*}
    \lim_{\rho\to -\infty}v'(\rho)=0, \quad \lim_{\rho\to\infty} v'(\rho)=q,\\
    \lim_{\rho\to -\infty}u'(\rho)=0, \quad \lim_{\rho\to\infty}u'(\rho)=b.
\end{align*}
The eigenvalues of $\alpha$ with respect to $\omega$ are 
\begin{equation*}
    \underbrace{\frac{v'}{u'}}_{m\text{-times}}, \, \underbrace{\frac{p+v'}{1+u'}}_{n\text{-times}}, \, \frac{v''}{u''}.
\end{equation*}
Since $u''>0$, we take $x=u'$ as a parameter. Define a function $y\in C^\infty([0,b])$ by $y(x)=v'$. Then, the eigenvalues can be written as
\begin{equation*}
    \underbrace{\frac{y}{x}}_{m\text{-times}}, \, \underbrace{\frac{p+y}{1+x}}_{n\text{-times}}, \, y'.
\end{equation*}
In summary,  the complex Hessian equation \eqref{eq: cHe'} of the form $\alpha\in p[\pi^*\omega_M]+q\eta$ is formulated as 
\begin{equation}\label{eq: ODE}
\begin{cases}
    \sigma_k\Bigl(\underbrace{\frac{y}{x}}_{m\text{-times}}, \underbrace{\frac{p+y}{1+x}}_{n\text{-times}}, y'\Bigr)\Big/{m+n+1\choose k}=c \sigma_\ell\Bigl(\underbrace{\frac{y}{x}}_{m\text{-times}}, \underbrace{\frac{p+y}{1+x}}_{n\text{-times}}, y'\Bigr)\Big/{m+n+1 \choose \ell},
    \\y(0)=0, \quad y(b)=q,
\end{cases}
\end{equation}
%where $\sigma_i$ is the $i$-th elementary symmetric function. 
Note that the equation in \eqref{eq: ODE} is a non-linear ODE of first order. Fang-Lai \cite{FL} observed that
$$\frac{d}{dx}G_{p}^{m,n,k}(x,y(x))=x^m(1+x)^n\sigma_k\Bigl(\underbrace{\frac{y}{x}}_{m\text{-times}}, \underbrace{\frac{p+y}{1+x}}_{n\text{-times}}, y'\Bigr),$$
where $G_{p}^{m,n,k}$ is defined by
$$G_{p}^{m,n,k}(x,y):=\frac{1}{k!}\frac{d^k}{dt^k}\Big|_{t=0}\int_0^{x+ty}s^m(1+tp+s)^n ds.$$
Note that $G_{p}^{m,n,k}$ is a polynomial of $x$ and $y$. 
Thus, the Dirichlet problem \eqref{eq: ODE} can be interpreted to
\begin{equation}\label{eq: poly}
    \begin{cases}
        G_{p}^{m,n,k}(x,y(x))=\mu G_{p}^{m,n,\ell}(x,y(x)), \\
        y(0)=0, \quad y(b)=q,
    \end{cases}
\end{equation}
where $\mu$ is the constant defined by $\mu:=c{m+n+1\choose k}/{m+n+1\choose \ell}$. Indeed, if we have \eqref{eq: poly}, we get \eqref{eq: ODE} by differentiating \eqref{eq: poly}. Conversely, if we have \eqref{eq: ODE}, we get the first equation of \eqref{eq: poly} by integrating the first equation of \eqref{eq: ODE} from 0 to $x$ and using $y(0)=0$. Note that the other boundary condition $y(b)=q$ does not violate the first equation of \eqref{eq: poly} by the following:
\begin{prop}\label{prop: boundary}
    $G_{p}^{m,n,k}(b,q)=\mu G_{p}^{m,n,\ell}(b,q)$.
\end{prop}
\begin{proof}
    We use the computations in \cite[Appendix A]{DMS}. For a class $\xi [\pi^*\omega_M]+\zeta \eta$, where $\xi, \, \zeta \in \mathbb{R}$, by \cite[Lemma A.3]{DMS}, we have
    $$(\xi [\pi^*\omega_M]+\zeta \eta)^{m+n+1}=d (m+n+1) {m+n+1\choose n}\sum_{i=0}^n {n\choose i}\frac{\xi^{n-i}\zeta^{m+1+i}}{m+i+1},$$
    where $d:= c_1(-L)^n$. On the other hand,
    \begin{align*}
        G_{\xi}^{m,n,m+n+1}(b,\zeta)
        &=\frac{1}{(m+n+1)!}\frac{d^{m+n+1}}{dt^{m+n+1}}\Big|_{t=0}\int_0^{b+t\zeta}s^m(1+t\xi+s)^n ds\\
        &=\frac{1}{(m+n+1)!}\frac{d^{m+n+1}}{dt^{m+n+1}}\Big|_{t=0}\int_0^{b+t\zeta}\sum_{i=0}^n{n\choose i}s^m(1+s)^i (t\xi)^{n-i} ds\\
        &=\sum_{i=0}^n{n\choose i} \frac{\zeta^{m+i+1}\xi^{n-i}}{m+i+1}
    \end{align*}
    Therefore, for $\varepsilon\in\mathbb{R}$, we have
    $$G_{1+\varepsilon p}^{m,n,m+n+1}(b,b+\varepsilon q)=C ([\omega]+\varepsilon [\alpha])^{m+n+1}$$
    for some constant $C$. For the right-hand side, we have
    $$\frac{1}{k!}\frac{d}{d\varepsilon^k}\Big|_{\varepsilon=0}([\omega]+\varepsilon[\alpha])^{m+n+1}={m+n+1\choose k}[\alpha]^k[\omega]^{m+n+1-k}.$$
    For the left-hand side, we have
    \begin{align*}
        &\frac{1}{k!}\frac{d}{d\varepsilon^k}\Big|_{\varepsilon=0}G_{1+\varepsilon p}^{m,n,m+n+1}(b,b+\varepsilon q)\\
        = &\frac{1}{k!}\frac{d}{d\varepsilon^k}\Big|_{\varepsilon=0}\frac{1}{(m+n+1)!}\frac{d^{m+n+1}}{dt^{m+n+1}}\Big|_{t=0}\int_0^{b+t(b+\varepsilon q)}s^m(1+t(1+\varepsilon p)+s)^n ds\\
        =&\frac{1}{k!}\frac{d}{d\varepsilon^k}\Big|_{\varepsilon=0} \sum_{i=0}^n {n\choose i}(1+\varepsilon p)^{n-i} \frac{(b+\varepsilon q)^{m+i+1}}{m+i+1}\\
        =&\frac{1}{k!}\frac{d}{dt^k}\Big|_{t=0} \sum_{i=0}^n {n\choose i}(1+t p)^{n-i} \frac{(b+t q)^{m+i+1}}{m+i+1}\\
        =&\frac{1}{k!}\frac{d}{dt^k}\Big|_{t=0}\int_0^{b+tq}s^m(1+tp+s)^n ds\\
        =&G_{p}^{m,n,k}(b,q).
    \end{align*}
    Therefore, we have
    \begin{equation}\label{eq: topological}
        G_{p}^{m,n,k}(b,q)=C {{m+n+1}\choose k}[\alpha]^k[\omega]^{m+n+1-k}.
    \end{equation}
    Since we defined $\mu$ as 
    $$\mu=\frac{{{m+n+1}\choose k} [\alpha]^k[\omega]^{{m+n+1-k}}}{{{m+n+1}\choose \ell}[\alpha]^\ell[\omega]^{m+n+1-\ell}},$$
    we get the assertion.
\end{proof}
In summary, solving the complex Hessian equation \eqref{eq: cHe'} in the class of $\alpha$-$(\omega,k)$-subharmonic functions is interpreted to solving the equation \eqref{eq: poly} with the condition
\begin{equation}\label{eq: gamma_k}
    \Bigl(\underbrace{\frac{y}{x}}_{m\text{-times}}, \, \underbrace{\frac{p+y}{1+x}}_{n\text{-times}}, \, y'\Bigr) \in \Gamma_k.
\end{equation}
Hereafter, 
for $\xi\in\mathbb{R}\setminus\{0\}$ and $\zeta\in\mathbb{R}$, we define
\begin{equation*}
   \check{\lambda}_{m,n}(\xi,\zeta):=\Bigl(\underbrace{\frac{\zeta}{\xi}}_{m\text{-times}}, \underbrace{\frac{p+\zeta}{1+\xi}}_{n\text{-times}}\Bigr).
\end{equation*}
For a function $y(x)$ such that $y(0)=0$, we also define
\begin{equation*}
    \lambda_{m,n}(x,y):=\Bigl(\underbrace{\frac{y}{x}}_{m\text{-times}}, \underbrace{\frac{p+y}{1+x}}_{n\text{-times}}, y'\Bigr), 
    \quad  \lambda_{m,n}(0,y(0)):=\Bigl(\underbrace{y'(0)}_{(m+1)\text{-times}}, \underbrace{p}_{n\text{-times}}\Bigr),
\end{equation*}
and $\check{\lambda}_{m,n}(0,y(0)):=\lambda_{m-1,n}(0,y(0))$.
In the rest of the paper, we omit $p$ of $G_p(x,y)$ since $p$ is fixed.

\subsection{The solution of the boundary value problem}\label{sec: $x=0$}
Now, we solve the boundary value problem \eqref{eq: poly} with condition \eqref{eq: gamma_k}, which concludes Theorem \ref{thm: main}. The proof is decomposed into three parts: near $x=0$, $0<x<b$, and at $x=b$. We first solve \eqref{eq: poly} near $x=0$. 

\begin{prop}\label{prop: near}
    Suppose Setup \ref{setup}. Let $[\alpha]$ be a cohomology class $p[\pi^*\omega_M]+q\eta$, where $p,q\in\mathbb{R}$.
    If the pair $([\alpha],[\omega])$ satisfies the numerical condition as in Conjecture \ref{conj} regarding the subvariety $P_0$, then there exist a constant $\varepsilon$ and a function $y\in C^\infty([0,\varepsilon))$ satisfying \eqref{eq: poly} \textup{(}without the condition $y(b)=q$\textup{)} with condition \eqref{eq: gamma_k}.
\end{prop}
\begin{rem}
    The numerical condition regarding $P_0$ can be explicitly written:
    \begin{itemize}
        \item If $m+1\le k-1$, the condition in Conjecture \ref{conj} \eqref{item: conj} implies $p>0$.
        \item If $m+1\le \ell$, the condition in Conjecture \ref{conj} \eqref{item: conj2} implies $$p^{k-m-1}{n\choose {k-m-1}}-\mu p^{\ell-m-1}{n\choose {\ell-m-1}}>0.$$
    \end{itemize}
\end{rem}

    Set $F_i(x,y)=G^{m-i,n,k-i}(x,y)-\mu G^{m-i,n,\ell-i}(x,y)$ for integers $i$ satisfying $0\le i \le m_0:=\min\{m,k-1\}$. When $\ell<i$, we define $G^{m-i,n,\ell-i}(x,y)=0$. We prove Proposition \ref{prop: near} by induction on $i$. We initially prove the following:
    \begin{lem}\label{lem: ini}
        Suppose the same assumptions as in Proposition \ref{prop: near}.
        \begin{enumerate}[font=\normalfont]
            \item If $m+1> k-1$, there exist a constant $\varepsilon$ and a function $y_{k-1}\in C^\infty([0,\varepsilon))$ such that $F_{k-1}(x,y_{k-1})=0$ and $y_{k-1}(0)=0$.
            \item If $m+1\le k-1$, there exist a constant $\varepsilon$ and a function $y_m\in C^\infty([0,\varepsilon))$ such that $F_m(x,y_m)=0$ and $y_m(0)=0$. Moreover, the function $y_m$ satisfies
            $\lambda_{0,n}(x,y_m),\check{\lambda}_{0,n}(x,y_m)\in\Gamma_{k-m-1}$ and $(\sigma_{k-m-1}-\mu\sigma_{l-m-1})(\check{\lambda}_{0,n}(x,y_m(x)))>0$ for $x\in[0,\varepsilon)$.
        \end{enumerate}
    \end{lem}
    \begin{proof}[Proof of Lemma \ref{lem: ini}]%\mbox{}
            \begin{enumerate}[wide,labelindent = 0pt, labelwidth = ! ]
            \item As stated in the previous section, %a function $y=y(x)$ satisfies 
            the equation $F_{k-1}(x,y)=0$ with the boundary condition $y(0)=0$ for a function $y=y(x)$ is equivalent to the following Dirichlet problem
            \begin{equation*}
                \begin{cases}
                    \sigma_1(\lambda_{m-k+1,n}(x,y))=\mu \sigma_{\ell-k+1}(\lambda_{m-k+1,n}(x,y)),
                    \\y(0)=0.
                \end{cases}
            \end{equation*}
            Note that the right-hand side is $\mu$ if $\ell=k-1$, and is $0$ if $\ell-k+1<0$ since we defined $G^{m-k+1,n,\ell-k+1}=0$ if $\ell-k+1<0$. Therefore the Dirichlet problem above is equivalent to 
            \begin{equation*}
                \begin{cases}
                    y'+\left(\frac{m-k+1}{x}+\frac{n}{1+x}\right)y+\frac{np}{1+x}-\mu\delta_{(k-1),\ell}=0\\
                    y(0)=0,
                \end{cases}
            \end{equation*}
            where $\delta_{i,j}$ is the Kronecker delta. By the standard technique of linear ODEs, this Dirichlet problem is solvable.\mbox{}
            \item Note that in general
            \begin{align*}
                \frac{d}{dy}G^{m,n,k}(x,y)
                &=\frac{1}{k!}\frac{d^k}{dt^k}\Big|_{t=0}x^m(1+x)^n\left(1+t\frac{y}{x}\right)^m\left(1+t\frac{p+y}{1+x}\right)^n t\\
                &=\frac{1}{(k-1)!}\frac{d^{k-1}}{dt^{k-1}}\Big|_{t=0}x^m(1+x)^n\left(1+t\frac{y}{x}\right)^m\left(1+t\frac{p+y}{1+x}\right)^n \\
                &=x^m(1+x)^n\sigma_{k-1}(\check{\lambda}_{m,n}(x,y)).
            \end{align*}
            Thus, we have
            \begin{equation*}
                \frac{d}{dy}F_m(x,y)=(1+x)^n\Bigl(\bigl(\sigma_{k-m-1}-\mu\sigma_{\ell-m-1}\bigr)(\check{\lambda}_{0,n}(x,y))\Bigr).
            \end{equation*}
            By assumption, we see that
            \begin{equation}\label{eq: ini}
                \frac{d}{dy}F_m(0,0)=p^{k-m-1}{n\choose{k-m-1}}-\mu p^{\ell-m-1}{n\choose{\ell-m-1}}>0.
            \end{equation}
            Hence, the implicit function theorem implies that there exists the solution of $F_m(x,y)=0$ near $(0,0)$. Since $p>0$ by assumption and $$\check{\lambda}_{0,n}(0,y_m(0))=(\underbrace{p,\dots,p}_{n\text{-times}}),$$ we have $\check{\lambda}_{0,n}(x,y_m(x))\in\Gamma_{k-m-1}$ for $x$ such that $0\le x\ll1$. The last condition of the statement also holds for $x$ such that $0\le x\ll1$ by \eqref{eq: ini}. To prove $\lambda_{0,n}(x,y_m(x))\in\Gamma_{k-m-1}$, note that
            \begin{align*}
                \sigma_i(\lambda_{0,n}(x,y_m(x)))
                =&-y'_m(x)\sigma_{i-1}(\check{\lambda}_{0,n}(x,y_m))+\sigma_i(\check{\lambda}_{0,n}(x,y_m))\\
                =&\left(-\frac{\sigma_{k-m}-\mu\sigma_{\ell-m}}{\sigma_{k-m-1}-\mu\sigma_{\ell-m-1}}\sigma_{i-1}+\sigma_i\right)(\check{\lambda}_{0,n}(x,y_m)),
            \end{align*}
        where the last equality follows from $(F_m(x,y_m(x))'=0$. 
        From the Newton inequality and the assumption $p>0$, we have
        \begin{equation*}
            \left(\frac{\sigma_{k-m}-\mu\sigma_{\ell-m}}{\sigma_{k-m-1}-\mu\sigma_{\ell-m-1}}\right)(\lambda)<\left(\frac{\sigma_{k-m}}{\sigma_{k-m-1}}\right)(\lambda)<\left(\frac{\sigma_i}{\sigma_{i-1}}\right)(\lambda)
        \end{equation*}
        for $\lambda:=\check{\lambda}_{0,n}(0,y_m(0))$ and $i$ such that $0\le i \le k-m-1$.
        Applying this inequality, we see that $\sigma_i(\lambda_{0,n}(0,y_m(0))>0$. Thus, if $x$ is small enough, we have $\lambda_{0,n}(0,y_m(0))\in\Gamma_{k-m-1}$.\qedhere
        \end{enumerate}
    \end{proof}
    We next clarify that induction works. First, we remark that if $p=0$ and $\ell-i+1<0$, the constant function $y_{i-1}(x)=0$ obviously satisfies $F_{i-1}(x,y_{i-1})=0$. Then, for the step $i-1=\ell$, by the same arguments as below, we get the function $y_\ell$ such that $F_\ell(x,y_\ell)=0$ and $y_\ell(0)=0$. This function $y_\ell$ also satisfies $y'_\ell(0)>0$, and therefore the conditions
    \begin{equation*}
        \lambda_{m-\ell,n}(x,y_\ell),\ \check{\lambda}_{m-\ell,n}(x,y_\ell)\in\Gamma_{k-i}.
    \end{equation*}
    hold for $x\in[0,\varepsilon)$ and some $\varepsilon>0$. After this step, we can run the same induction below.
    For the case $p\neq0$, the following lemma claims that induction works:
    
    \begin{lem}\label{lem: ind}
        Assume $p\neq0$. Suppose there exists a function $y_i(x)\in C^\infty([0,\varepsilon))$ for some $\varepsilon>0$ such that 
        \begin{align}
                &F_i(x,y_i)=0, \quad y_i(0)=0,\notag \\
                &\lambda_{m-i,n}(x,y_i), \ \check{\lambda}_{m-i,n}(x,y_i) \in\Gamma_{k-i-1} \ \text{ for }\ x\in[0,\varepsilon),\notag\\
                &\label{eq: ass}(\sigma_{k-i-1}-\mu\sigma_{\ell-i-1})(\check{\lambda}_{m-i,n}(0,y_i(0)))>0 \ \text{ if }\ \ell-i-1\ge0.
        \end{align}
        Then, there exists a function $y_{i-1}\in C^\infty([0,\varepsilon'))$ for some $\varepsilon'>0$ such that 
        \begin{align}
                &F_{i-1}(x,y_{i-1})=0, \quad y_i(0)=0,\notag\\
                &\lambda_{m-i+1,n}(x,y_{i-1}), \ \check{\lambda}_{m-i+1,n}(x,y_{i-1}) \in\Gamma_{k-i-1}
                \ \text{ for }\ x\in[0,\varepsilon'),\notag\\
                &\label{eq: con}(\sigma_{k-i}-\mu\sigma_{\ell-i})(\check{\lambda}_{m-i+1,n}(0,y_i(0)))>0 \ \text{ if }\ \ell-i\ge0.
        \end{align}
        \end{lem}
    \begin{proof}
       First, we prove that $F_{i-1}(x,y_i(x))<0$ near $x=0$. Note that
        \begin{align*}
            &\frac{d}{dx}F_{i-1}(x,y_i(x))\\
            =&x^{m-i+1}(1+x)^n\Bigl(\bigl(\sigma_{k-i+1}-\mu\sigma_{\ell-i+1}\bigr)(\lambda_{m-i+1,n}(x,y_i))\Bigr)\\
            =&x^{m-i+1}(1+x)^n\biggl(\Bigl(y'_i(\sigma_{k-i}-\mu\sigma_{\ell-i})
            +(\sigma_{k-i+1}-\mu\sigma_{\ell-i+1})\Bigr)(\check{\lambda}_{m-i+1,n}(x,y_i))
            \biggr).
        \end{align*}
        As $x\to0$, by $(F_i(x,y_i(x)))'=0$, we get
        $\left(\sigma_{k-i}-\mu\sigma_{\ell-i}\right)(\lambda_{m-i,n}(0,y_i(0)))=0.$
        We further claim that
        $(\sigma_{k-i+1}-\mu\sigma_{\ell-i+1})(\lambda_{m-i,n}(0,y_i(0)))<0.$ (This inequality does not hold if $p=0$ and $\ell-i+1<0$.)
        Indeed, suppose not, namely, the left-hand side is nonnegative. Then, if $\ell-i+1<0$, we have $\sigma_{k-i+1}(\lambda_{m-i,n}(0,y_i(0)))\ge0$. On the other hand, by the definition of $y_i$, we have $\sigma_{k-i}(\lambda_{m-i,n}(0,y_i(0)))=0$. 
        By the assumption of induction and this inequality, we get $\sigma_{k-i+1}(\lambda_{m-i,n}(0,y_i(0)))=0$, and since the equality in the Newton inequality holds, we have $p=y'_i(0)$. However, this conclusion contradicts to $\sigma_{k-i}(\lambda_{m-i,n}(0,y_i(0)))=0$ since $p\neq0$. Thus, we get $(\sigma_{k-i+1}-\mu\sigma_{\ell-i+1})(\lambda_{m-i,n}(0,y_i(0)))<0$. 
        If $\ell-i+1\ge0$, by assumption, we have $(\sigma_{k-i+1}-\mu\sigma_{\ell-i+1})(\lambda_{m-i,n}(0,y_i(0)))\ge0$. In particular, we have $\sigma_{k-i+1}(\lambda_{m-i,n}(0,y_i(0)))>0$ and therefore $\lambda_{m-i,n}(0,y_i(0))\in\Gamma_{k-i+1}$. On the other hand, by the definition of $y_i$, we have $(\sigma_{k-i}-\mu\sigma_{\ell-i})(\lambda_{m-i,n}(0,y_i(0)))=0$. Since $\lambda_{m-i,n}(0,y_i(0))\in\Gamma_{k-i+1}$, by the Newton inequality, we get 
        \begin{align*}
            \mu
            \le\left(\frac{\sigma_{k-i+1}}{\sigma_{\ell-i+1}}\right)(\lambda_{m-i,n}(0,y_i(0)))
            <\left(\frac{\sigma_{k-i}}{\sigma_{\ell-i}}\right)(\lambda_{m-i,n}(0,y_i(0)))=\mu,
        \end{align*}
        which is a contradiction. Therefore, we get $(\sigma_{k-i+1}-\mu\sigma_{\ell-i+1})(\lambda_{m-i,n}(0,y_i(0)))<0$. As a result, we have $(F_{i-1}(x,y_i(x)))<0$ near $x=0$. Next, we prove that $(D_y^2 F_{i-1})(x,y)>0$ for $y\ge y_i(x)$ and $x\neq0$ sufficiently small, where $D_y$ denotes the derivation with respect to $y$. Let $\ell_0:=\max\{2,\ell-i+2\}$. Then, we have
        \begin{align*}
             &(D_y^{\ell_0}F_{i-1})(x,y)\\
            =&D_y^{\ell_0-1}\Bigl(x^{m-i+1}(1+x)^n \bigl(\sigma_{k-i}-\mu\sigma_{\ell-i}\bigr)(\check{\lambda}_{m-i+1,n}(x,y))
            \Bigr)\\
            =&\sum_{r=\max\{0,-n+\ell_0-1\}}^{\min\{\ell_0-1,m-i+1\}}C_r x^{m-i+1-r}(1+x)^{n-\ell_0+1+r}
            \sigma_{(k-i)-(\ell_0-1)}(\check{\lambda}_{m-i+1-r,n-\ell_0+1+r}(x,y)),
        \end{align*}
        where $C_r$'s are some positive constants.
        Since $\check{\lambda}_{m-i,n}(x,y_i)\in \Gamma_{k-i-1}$, by using the monotonicity of $\sigma_s$ in $\Gamma_s$, the terms corresponding to $r\neq 0$ are positive for any $y\ge y_i(x)$. For the term $r=0$, we have
        \begin{align*}
            &\sigma_{(k-i)-(\ell_0-1)}(\check{\lambda}_{m-i+1,n-\ell_0+1}(x,y))\\
            \ge&\sigma_{(k-i)-(\ell_0-1)-1}(\check{\lambda}_{m-i,n-\ell_0+1}(x,y))
            \left(\frac{y_i}{x}+\left(\frac{\sigma_{(k-i)-(\ell_0-1)}}{\sigma_{(k-i)-(\ell_0-1)-1}}\right)(\check{\lambda}_{m-i,n-\ell_0+1}(x,y_i))\right),
        \end{align*}
        where we used the assumption $\check{\lambda}_{m-i,n}(x,y_i)\in\Gamma_{k-i-1}$, the monotonicity of $\sigma_{r}/\sigma_{r-1}$ in $\Gamma_r$, and $y\ge y_i(x)$. Using $(F_i(x,y_i))'=0$, the latter part of the right hand side converges to
        \begin{equation*}
            -\left(\frac{\sigma_{k-i}-\mu\sigma_{\ell-i}}{\sigma_{k-i-1}-\mu\sigma_{\ell-i-1}}\right)(\check{\lambda}_{m-i,n}(0,y_i(0)))
            +\left(\frac{\sigma_{(k-i)-(\ell_0-1)}}{\sigma_{(k-i)-(\ell_0-1)-1}}\right)(\check{\lambda}_{m-i,n-\ell_0+1}(0,y_i(0))).
        \end{equation*}
        We claim that this quantity is always positive. Indeed, if $\ell-i<0$, we have $\ell_0=2$, and
        \begin{align*}
            &\sigma_{k-i-1}(\check{\lambda}_{m-i,n})\sigma_{k-i-1}(\check{\lambda}_{m-i,n-1})-\sigma_{k-i}(\check{\lambda}_{m-i,n})\sigma_{k-i-2}(\check{\lambda}_{m-i,n-1})\\
            =&\Bigl(p\sigma_{k-i-2}(\check{\lambda}_{m-i,n-1})+\sigma_{k-i-1}(\check{\lambda}_{m-i,n-1})\Bigr)\sigma_{k-i-1}(\check{\lambda}_{m-i,n-1})\\
            &\quad-\Bigl(p\sigma_{k-i-1}(\check{\lambda}_{m-i,n-1})+\sigma_{k-i}(\check{\lambda}_{m-i,n-1})\Bigr)\sigma_{k-i-2}(\check{\lambda}_{m-i,n-1})\\
            =&\Bigl(\sigma^2_{k-i-1}-\sigma_{k-i}\sigma_{k-i-2}\Bigr)(\check{\lambda}_{m-i,n-1})>0,
        \end{align*}
        where we omit $(0,y_i(0))$ for simplicity and the last inequality is the Newton inequality. Since $\check{\lambda}_{m-i,n}(0,y_i(0))\in\Gamma_{k-i-1}$, from this inequality, we see that the quantity is positive. If $\ell-i\ge0$, by the Newton inequality and the assumptions $\check{\lambda}_{m-i,n}(0,y_i(0))\in\Gamma_{k-i-1}$ and \eqref{eq: ass}, we see that
        \begin{equation*}
            \left(\frac{\sigma_{k-i}-\mu\sigma_{\ell-i}}{\sigma_{k-i-1}-\mu\sigma_{\ell-i-1}}\right)(\check{\lambda}_{m-i,n}(0,y_i(0)))
            <\left(\frac{\sigma_{k-i}}{\sigma_{k-i-1}}\right)(\check{\lambda}_{m-i,n}(0,y_i(0))).
        \end{equation*}
        Then, we see that the quantity is positive by repeatedly applying the same calculation as in the case $\ell-i<0$. As a result, we have proved that $D^{\ell_0}_y F_{i-1}(x,y)$ is positive for $x$ and $y$ such that $0\neq x\ll1$ and $y\ge y_i(x)$. On the other hand, for $s\ge2$, we have
        \begin{align*}
            (D^s_y F_{i-1})(x,y_i(x))
            =&\sum_{r=\max\{0,-n+s-1\}}^{\min\{s-1,m-i+1\}}C_r x^{m-i+1-r}(1+x)^{n-s+1+r}\\
            &\qquad\times\Bigl(\sigma_{(k-i)-(s-1)}-\mu\sigma_{(\ell-i)-(s-1)}\Bigr)(\check{\lambda}_{m-i+1-r,n-s+1+r}(x,y_i)).
        \end{align*}
        Since $\check{\lambda}_{m-i,n}\in\Gamma_{k-i-1}$ and \eqref{eq: ass} holds by assumption, from the Newton inequality, we see that the terms corresponding to $r\neq0$ are positive. For the term corresponding to $r=0$, note that $$\lambda_{m-i,n-s+1}(0,y_i(0))=\check{\lambda}_{m-i+1,n-s+1}(0,y_i(0))$$
        by definition and $$\Bigl(\sigma_{k-i}-\mu\sigma_{\ell-i}\Bigr)(\lambda_{m-i,n}((0,y_i(0))=0.$$
        Since $\sigma_r/\sigma_s$ is strictly increasing in $\Gamma_r$, where $r$ and $s$ satisfy $r>s$, by the assumption $\lambda_{m-i,n}(0,y_i(0))\in\Gamma_{k-i-1}$, taking some components to infinity, we get
        $$\left(\frac{\sigma_{(k-i)-(s-1)}}{\sigma_{(\ell-i)-(s-1)}}\right)(\lambda_{m-i,n-s+1}(0,y_i(0)))<\left(\frac{\sigma_{k-i}}{\sigma_{\ell-i}}\right)(\lambda_{m-i,n}(0,y_i(0)))=\mu.$$
        Therefore, for $x$ small enough, the term corresponding to $r=0$ is also positive. As a result, we proved that $(D^s_y F_{i-1})(x,y_i(x))>0$ for $x$ small enough. Combining two results $(D^{\ell_0}_y F_{i-1}) (x,y)>0$ for $y\ge y_i(x)$ and $(D^s_y F_{i-1})(x,y_i(x))>0$ for $s\ge2$, we have $(D^2_y F_{i-1})(x,y)>0$ for $y\ge y_i(x)$. Combining this result with $F_{i-1}(x,y_i(x))<0$, we obtain a function $y_{i-1}(x)$, which satisfies $y_{i-1}(x)> y_i(x)$, for any $x>0$ small enough. Moreover, since $y_i(0)=0$, we have $y_{i-1}(0)\ge0$. On the other hand, since $i \le m$, for any $y>0$, if $x>0$ is small enough, 
        \begin{equation*}
            (D_xF_{i-1})(x,y)=x^{m-i+1}(1+x)^n\left(\sigma_{k-i+1}-\mu\sigma_{\ell-i+1}\right)(\check{\lambda}_{m-i+1,n}(x,y))>0.
        \end{equation*}
        Hence, we have $y_{i-1}(0)=0$.
        Since $y_{i-1}(x)>y_i(x)$, we get $y'_{i-1}(0)\ge y'_i(0)$. Since $y'_i(0)$ satisfies 
        $$\Bigl(\sigma_{k-i+1}-\mu\sigma_{l-i+1}\Bigr)(\lambda_{m-i+1,n}(0,y_i(0))<0$$
        as we proved in the proof of the first claim, we must have $y'_{i-1}(0)>y'_i(0)$. In particular, since 
        $$\Bigl(\sigma_{k-i}-\mu\sigma_{\ell-i}\Bigr)(\lambda_{m-i,n}(0,y_i(0))=0,$$
        condition $\eqref{eq: con}$ holds by using the strict monotonicity of $\sigma_{k-i}/\sigma_{\ell-i}$. The condition $\lambda_{m-i+1,n}(0,y_{i-1}(0))\in\Gamma_{k-i}$ follows from the same arguments as in the proof of Lemma \ref{lem: ini}. Therefore, we have $(D_y F_{i-1})(x,y_{i-1})>0$ for $x>0$ sufficiently small, and the implicit function theorem thus implies the smoothness of $y_{i-1}$ in $(0,\varepsilon')$. Condition \eqref{eq: con} also implies that higher derivatives of $y_{i-1}$ at $0$ do not diverge, namely $y_{i-1}\in C^\infty([0,\varepsilon')$. The condition $\check{\lambda}_{m-i+1,n}(0,y_{i-1}(0))\in\Gamma_{k-i}$ follows from $\lambda_{m-i,n}(0,y_i(0))\in\Gamma_{k-i-1}$ and \eqref{eq: con}.
    \end{proof}
    In conclusion, induction works by Lemmas \ref{lem: ini} and \ref{lem: ind}, and eventually, we get a solution $y=y_0(x)$ of \eqref{eq: poly} near $x=0$ (without the condition $y(b)=q$) with condition \eqref{eq: gamma_k}. Thus, we have proved Proposition \ref{prop: near}.

Next, we prove that the solution can be extended arbitrarily longer with condition \eqref{eq: gamma_k}. 
Let $y(x)$ be the solution near $x=0$ we obtained by Proposition \ref{prop: near}. Let us define
$$T:=\sup\left\{ t\in[0,b]\ \middle\vert 
\begin{array}{l} y(x) \text{ can be extended to a solution of \eqref{eq: poly} on $[0,t]$ }\\
\text{(without the condition $y(b)=q$) with condition \eqref{eq: gamma_k}} 
\end{array}\right\}.
$$

\begin{prop}\label{prop: interval}
     $T=b$.
\end{prop}
\begin{proof}
    By Proposition \ref{prop: near}, we know that $T>0$. Suppose $T<b$ and take a monotone sequence $\{x_i\}\subset [0,T)$ such that $x_i\to T$. Note that $F(x,y)=0$ is a polynomial of $y$ with coefficients in $\mathbb{Z}[x]$, that these coefficients are bounded as $0<x<b$, and that the coefficient of the highest degree is not zero for $0<x<b$. Therefore, we have a uniform bound of $\{y(x_i)\}$. After taking a subsequence, we can assume that $y(x_i)\to y(T)$ for some $y(T)$. It suffices to prove $(D_yF)(T,y(T))>0$. Suppose not. Then, either $y'(x_i)\to \pm\infty$ (the case $-\infty$ can be eliminated since $\lambda_{m,n}(x,y(x))\in\Gamma_k\subset \Gamma_1$ and $y(x)$ is bounded). Each case implies that 
    $$\left(\frac{y(x)}{x}\right)'\Big|_{x=x_i}, \ \left(\frac{1+y(x)}{1+x}\right)'\Big|_{x=x_i}\to \pm\infty \quad \text{as $i\to \infty$}.$$
    However, this is impossible since $(\sigma_k-\mu\sigma_\ell)(\lambda_{m,n}(x,y(x))=0$ and $\sigma_k/\sigma_\ell$ is strictly increasing in $\Gamma_k$. Therefore, we have $T=b$.
\end{proof}

Lastly, we prove that the function $y_0:=y(x)\in C^\infty([0,b])$ satisfies the boundary condition $y_0(b)=q$. When $\ell=0$, this equality follows from the numerical positivity for the subvarieties coming from cutting by $\eta=c_1(\mathcal{O}(1))$. When $\ell>0$ and there exists a smooth $\alpha$-$(\omega,k)$-subharmonic function, the equality follows from the numerical positivity for the subvariety $D_\infty$.

\begin{prop}
    If the pair $([\alpha],[\omega])$ satisfies the numerical criterion in Conjecture \ref{conj}, then we have $y_0(b)=q$.
\end{prop}
\begin{proof}
     By Proposition \ref{prop: boundary}, we have $F(b,q)=0$. By Proposition \ref{prop: interval}, we have $(D_yF)(b,y_0(b))>0$. Also, by the Newton inequality, we see that $\check{\lambda}_{m,n}(b,y_0(b))\in\Gamma_{k-1}$. Thus, we have $(D_yF)(b,y)>0$ for any $y$ such that $y\ge y_0(b)$. It suffices to prove $(D_yF)(b,y)>0$ for any $y$ such that $y\ge q$. By differentiating the equality \eqref{eq: topological} with respect to $q$, we get 
    $$(D^s_yF)(b,y)=C\int_{D_\infty^s}\Bigl(\alpha_y^{k-s}\wedge\omega^{m+n+1-k}-c\alpha_y^{\ell-s}\wedge\omega^{m+n+1-\ell}\Bigr),$$
    where $\alpha_y$ is a closed real $(1,1)$-form whose cohomology class is $p[\pi^*\omega_M]+y\eta$ and $D_\infty^s$ is a subvariety associated to cutting $s$-times by $[D_\infty]=\eta$. If $\ell=0$, since $\omega$ is a K{\"a}hler form, we get $(D^k_yF)(b,y)>0$ for any $y$. The numerical condition in Conjecture \ref{conj} implies $(D^s_yF)(b,q)>0$ for an integer $s$ satisfying $1\le s\le k-1$. Hence, we see that $(D_yF)(b,y)>0$ for any $y\ge q$. If $\ell>0$, from the case $\ell=0$, we see that $\check{\lambda}_{m,n}(b,q)\in\Gamma_{k-1}$. Thus, if $(D_yF)(b,q)>0$, we see that $(D_yF)(b,y)>0$ for any $y\ge q$, since $\sigma_{k-1}/\sigma_{\ell-1}$ is strictly increasing in $\Gamma_{k-1}$. The condition $(D_yF)(b,q)>0$ is nothing but the numerical condition for $D_\infty$. 
\end{proof}

\section{Proof of Theorem \ref{thm: semiample}}\label{sec: semi}
The proof of Theorem \ref{thm: semiample} is almost the same as arguments in \cite[Section 3]{Mat}, except for a small adjustment to $\alpha$-$(\omega,k)$-subharmonicity. Let $X$ be an $n$-dimensional compact K{\"a}hler manifold.

\begin{proof}[Proof of Theorem \ref{thm: semiample}]
    Since $L$ is a semiample line bundle, there exists the associated Iitaka fibration $\Phi: X\to Y$.
    Since $L$ is a trivial line bundle on a generic fiber $\Phi$, the numerical condition in Conjecture \ref{conj} \eqref{item: conj} implies that the generic fiber has a dimension less than $n-k$. Let $\omega_Y\in c_1(L_Y)$, where $L_Y$ is an ample line bundle on $Y$ such that $\Phi^*L_Y=L$. As in \cite[Definition 3.2]{Mat}, Let $B$ be the degenerated locus of $\alpha:=\Phi^*\omega_Y$, i.e.,
    \begin{equation*}
        B:=\{p\in X \mid\alpha\text{ has at least $(n-k+1)$ zero eigenvalues at $p$}\}.
    \end{equation*}
    According to \cite[Proposition 3.3]{Mat}, the degenerated locus $B$ is a closed analytic set on $X$. According to \cite[Proposition 3.6]{Mat}, there exist sets $D_s (0\le s \le \dim B)$ with the following properties:
    \begin{enumerate}
        \item $D_s$ is an analytic (possibly, not closed, not irreducible, singular) set on $X$.
        \item \label{item: filt} $D_{\dim B}:=B\supset D_{\dim B-1}\supset\dots\supset D_1\supset D_0$.
        \item $\dim D_s=s$ for $s=0,1,2,\dots,\dim B$.
        \item $D_s\setminus D_{s-1}$ for $s=1,2,\dots,\dim B$ is a disjoint union of non-singular analytic sets.
        \item\label{item: alpha} For an irreducible component $W$ of $D_s\setminus D_{s-1}$ with $\dim W>n-k$, the $(1,1)$-form $\alpha\mid_W$ has $(\dim W-n +k)$ positive eigenvalues.
    \end{enumerate}
    Accoridng to \cite[Proposition 3.7]{Mat}, for $s=0,1,\dots, \dim B$, there exists a function $\varphi_s\in C^2(X,\mathbb{R})$ with the following properties:
    \begin{enumerate}
        \item\label{item: posi}  Let $W$ be an irreducible component of $D_s\setminus D_{s-1}$. Then the Levi form $\delb \varphi_s$ has $(n-\dim W)$ positive eigenvalues in the normal direction of $W$.
        \item\label{item: semi} The form $\delb \varphi_s$ is semi-positive at every point in $\overline{D_s}$.
    \end{enumerate}
    Fix a K{\"a}hler form $\omega$ on $X$.
    We claim that for any $x\in X$ there exists a constant $\varepsilon(x)>0$ such that for any $0<\varepsilon<\varepsilon(x)$, the form 
    \begin{equation}\label{eq: form}
        \alpha+\delb\left(\sum_{a=0}^{\dim B}\varepsilon^{\sum_{b=0}^a k^b}\varphi_{\dim B-a}\right)
    \end{equation}
    is $(\omega,k)$-subharmonic at $x$, i.e., the eigenvalue vector of this form with respect to $\omega$ at $x$ is in $\Gamma_k$. If this is true, from the compactness of $X$, the theorem is proved. Now we prove the claim. Suppose $x\in X\setminus B$. Then, by the definition of $B$, the form $\alpha$ has $k$ positive eigenvalues and $(\dim X-k)$ semi-positive eigenvalues at $x$. Thus, we have $\sigma_i(\lambda(\alpha\cdot\omega^{-1}))>0$ for $i=1,2,\dots,k$, namely $\lambda(\alpha\cdot\omega^{-1})\in\Gamma_k$, where $\lambda(\alpha\cdot\omega^{-1})$ is the eigenvalue vector of $\alpha$ with respect to $\omega$. Since form \eqref{eq: form} converges to $\alpha$ as $\varepsilon\to0$, there exists a constant $\varepsilon(x)$ such that if $\varepsilon<\varepsilon(x)$ then form \eqref{eq: form} is $(\omega,k)$-subharmonic at $x$. Next, suppose $x\in D_s\setminus D_{s-1}$ for $s=0,1,\dots,\dim B$. 
    By property \eqref{item: alpha} of $D_s$ and properties \eqref{item: posi} and \eqref{item: semi} of $\varphi_s$, the form
    \begin{equation}\label{eq: form2}
        \alpha+\delb\left(\varepsilon^{\sum_{b=0}^{\dim B-s} k^b}\varphi_s\right)
    \end{equation}
    has $k$ positive eigenvalues and $\dim X-k$ semipositive eigenvalues for any $\varepsilon>0$. Thus, its eigenvalue vector is in $\Gamma_k$. 
    On the other hand, take a coordinate neighborhood around $x$ such that $\alpha=\sum c_i \sqrt{-1}dz^i\wedge d\bar{z}^i$ at $x$, where $c_i$'s are one or zero, and $\delb\varphi_s=\sum \lambda_i \sqrt{-1}dz^i\wedge d\bar{z}^i$ at $x$. By property \eqref{item: filt} of $\{D_s\}$ and property \eqref{item: semi} of $\{\varphi_s\}$, the $i$-th eigenvalue of form \eqref{eq: form} can be estimated from below by
    $$\tilde{\lambda}_i:=c_i+\varepsilon^{\sum_{b=0}^{\dim B-s}k^b}\lambda_i-\varepsilon^{\sum_{b=0}^{\dim B-s+1}k^b}C,$$
    for some $C>0$ at $x$. For $j=1,2,\dots,k$, we have
    \begin{align*}
    \sigma_j\bigl((\tilde{\lambda}_i)_i\bigr)=\sigma_j\bigl((c_i+\varepsilon^{\sum_{b=0}^{\dim B-s}k^b}\lambda_i)_i\bigr)+O(\varepsilon^{\sum_{b=0}^{\dim B-s+1}k^b}).
    \end{align*}
    Note that the first term of the right-hand side has at most $(\varepsilon^{\sum_{b=1}^{\dim B-s+1}k^b})$-order since $j\le k$ and is positive since the eigenvalue vector of form \eqref{eq: form2} is in $\Gamma_k$. Therefore, by using the monotonicity of $\sigma_j$ in $\Gamma_j$, we see that the eigenvalue vector of form \eqref{eq: form} is in $\Gamma_k$ for sufficiently small $\varepsilon>0$. Hence, we verified the claim and proved the theorem.
\end{proof}

\appendix
\section{Relation to $q$-positivity in the sense of Andreotti-Grauert}
Let $X$ be an $n$-dimensional compact K{\"a}hler manifold. In this appendix, we make several remarks on the relation between $(\omega,k)$-subharmonicity and $(n-k)$-positivity in the sense of Andreotti-Grauert. 
Recall the definition of $q$-positivity:
\begin{defn}[$q$-positivity]
    A holomorphic line bundle $L$ on $X$ is called $q$-positive if there exists a smooth Hermitian metric $h$ whose Chern curvature $\sqrt{-1}\Theta_h$ has at least $(n-q)$ positive eigenvalues at any point on $X$ as a $(1,1)$-form.
\end{defn}
Yang \cite[Proposition 2.2]{Yang} gives the following characterization of $q$-positivity:

\begin{prop}[{\cite[Proposition 2.2]{Yang}}]\label{prop: Yang}
    The following conditions are equivalent:
    \begin{enumerate}[font=\normalfont]
        \item A holomorphic line bundle $L$ is $q$-positive.
        \item There exist a smooth hermitian metric $h$ and a smooth hermitian form $\omega$ on $X$ such that the summation of any distinct $(q+1)$ eigenvalues $($counting multiplicity$)$ of the Chern curvature $\sqrt{-1}\Theta_h$ of $h$ is positive at any point of $X$ $($with respect to $\omega)$.
    \end{enumerate}
\end{prop}
The latter condition is called ``uniform $q$-positivity'' in \cite[Definition 2.1]{Yang}. It is also related to ``$(q+1)$-plurisubharmonicity'' in the sense of Harvey-Lawson's geometric potential theory \cite[p434--435]{HL}.

Since $\sigma_i$ is striclty increasing in $\Gamma_i$ as in Proposition \ref{prop: mono}, if the Chern curvature $\sqrt{-1}\Theta_h$ is $(\omega,k)$-subharmonic for some smooth hermitian form $\omega$, the summation of any distinct $(n-k+1)$ eigenvalues (conting multiplicity) of $\sqrt{-1}\Theta_h$ is positive at any point of $X$ with respect to $\omega$. Hence, $(\omega,k)$-subharmonicity implies $(n-k)$-positivity. 

On the other hand, we can prove a converse implication:
\begin{prop}
    If $L$ is $q$-positive, there exist a smooth hermitian metric $h$ of $L$ and a smooth hermitian form $\omega$ on $X$ such that the Chern curvature form $\sqrt{-1}\Theta_h$ is $(\omega,n-q)$-subharmonic.
\end{prop}
\begin{proof}
    We prove the statement by induction. Suppose $q=n-1$. Then, the uniform $q$-positivity is nothing but the $(\omega,1)$-subharmonicity. So the statement holds. Suppose now that the statement holds for any integers strictly larger than $q$. Since the $q$-positivity implies $(q+1)$-positivity, by the assumption of induction, there exist a smooth hermitian metric $h$ and a smooth hermitian form $\tilde{\omega}$ on $X$ such that the Chern curvature $\sqrt{-1}\Theta_h$ is $(\tilde{\omega},n-q-1)$-subharmonic.
    Let us fix a point $p\in X$. Take a coordinate neighborhood $U_p$ of $p$ such that $\tilde{\omega}(p)=\sum\sqrt{-1}dz^i\wedge d\bar{z}^i$ and $R(p)=\sum\lambda_i\sqrt{-1}dz^i\wedge d\bar{z}^i$, where $\lambda_1\ge\lambda_2\ge\dots\ge\lambda_n$. Note that
    \begin{equation*}
        \sigma_{n-q}(\lambda)=\lambda_1\sigma_{n-q-1}(\check{\lambda})+\sigma_{n-q}(\check{\lambda}),
    \end{equation*}
    where $\lambda=(\lambda_1,\lambda_2,\dots,\lambda_n)$ and $\check{\lambda}=(\lambda_2,\dots,\lambda_n)$. Since $\sigma_{n-q-1}(\check{\lambda})>0$ by the assumption of induction, taking a hermitian form $\Omega_p$ on $U_p$ whose corresponding metric measures the eigenvector for $\lambda_1$ small enough at $p$, we can get $\sigma_{n-q}(\tilde{\lambda})>0$ at $p$, where $\tilde{\lambda}$ is the eigenvalue vector of $R\Omega^{-1}_p$. By taking a sufficiently smaller neighborhood $W_p$ than $U_p$, we can get $\sigma_{n-q}(\tilde{\lambda})>0$ on $W_p$. By compactness, there exist a finite number of points $\{p_i\}_{i=1}^N$ such that $\{W_{p_i}\}$ is an open cover of $X$. Denote a partition of unity subordinate to $\{W_{p_i}\}$ by $\{f_i\}$. Define the hermitian form $\omega$ by the equation
    $$\omega^{-1}=\sum_i f_i\Omega^{-1}_{p_i}.$$
    Then, by the convexity of $(\sigma_i)^{1/i}$ in $\Gamma_i$ (see \cite[Corollary 2.4]{Spr}), we get
    \begin{align*}
        \Bigl(\sigma_{n-q}(R\omega^{-1})\Bigr)^{\frac{1}{n-q}}
        =&\left(\sigma_{n-q}(\sum_i f_i R\Omega^{-1}_{p_i}))\right)^{\frac{1}{n-q}}\\
        \ge&\sum_i f_i \left(\sigma_{n-q}(R\Omega^{-1}_{p_i})\right)^{\frac{1}{n-q}}>0.\qedhere
    \end{align*}
\end{proof}

Finally, we remark that a stronger implication holds for semiample line bundles. Recall that we only needed the generic fiber of the Iitaka fibration associated with a semiample line bundle $L$ to have a dimension less than $n-k$ in the proof of Theorem \ref{thm: semiample}. \cite[Theorem 1.4]{Mat} claims that this condition is equivalent to $(n-k)$-positivity of $L$. In conclusion, for any fixed hermitian form $\omega$, the $(n-k)$-positivity of a semiample line bundle $L$ implies there exists a smooth hermitian metric $h$ whose Chern curvature form is $(\omega,k)$-subharmonic.

\end{document}